\documentclass[a4paper,12pt]{article}
\usepackage[latin1]{inputenc}
\usepackage{amssymb}
\usepackage{graphicx}
\usepackage{amsmath}
\usepackage{calc}
\usepackage{eurosym}
\usepackage{anysize}
\usepackage{amsthm}
\usepackage{enumerate}
\usepackage{semtrans}
\usepackage{eufrak}
\usepackage{mathrsfs}
\usepackage{yfonts}
\usepackage{bbold}
\usepackage{color}
\usepackage{pst-node}
\usepackage{pstricks}
\usepackage{pstricks-add,pst-plot,pst-node}
\usepackage[nottoc, notlof, notlot]{tocbibind}
\usepackage{wasysym}
\usepackage{MnSymbol}
\marginsize{2,5cm}{2,5cm}{2,2cm}{2,3cm}

\title{Urns with simultaneous drawing}
\date {\today}
\author{Micka{\"{e}}l Launay}

\definecolor{jaune} {cmyk}{0,0.4,1,0}%
\definecolor{bleu} {cmyk}{1,0,0,0}%
\definecolor{darkred} {rgb}{.4,0,0}%
\definecolor{darkblue} {rgb}{0,0,.4}%
\definecolor{lightred} {rgb}{1,.7,.7}%
\definecolor{lightblue} {rgb}{.7,.7,1}%

\definecolor{redb} {rgb}{.5,1,.1}%
\definecolor{greenb} {rgb}{.5,.0,.0}%

\newcommand{\N}{\mathbb{N}}

\newcommand{\R}{\mathbb{R}}
\newcommand{\Z}{\mathbb{Z}}

\newcommand{\Prob}{\mathbb{P}}
\newcommand{\E}{\mathbb{E}}

\begin{document}

\newtheorem{lemme}{Lemma}[section]
\newtheorem{theorem}[lemme]{Theorem}
\newtheorem{prop}[lemme]{Proposition}
\newtheorem{coro}[lemme]{Corollaire}
\newtheorem{conj}[lemme]{Conjecture}
\newtheorem{defi}[lemme]{Definition}
\newtheorem{ex}[lemme]{Exemples}

\maketitle

\begin{abstract}
In classical urn models, one usually draws one ball with replacement at each time unit and then adds one ball of the same colour. Given a weight sequence $(w_k)_{k\in\N}$, the probability of drawing a ball of a certain colour is proportional to $w_k$ where $k$ is the number of balls of this colour. A classical result states that an urn fixates on one colour after a finite time if an only if $\sum_{0}^\infty w_k^{-1} < \infty$. In this paper we shall study the case when at each time unit we draw with replacement a number $d\in\N$ of balls and then add $d$ new balls of matching colours. The main goal is to prove that the result in the case of maximal interaction generalizes assuming in addition that $(w_k)_{k\in\N}$ is non-decreasing.
\end{abstract}

\section{Introduction}

\subsection{The model}

The history of reinforced urn processes starts in 1930, when G. P\'{o}lya introduced a now famous model to study the spread of an epidemic \cite{polya}. The reader curious of the multitude of different urn processes that have been studied in the last decades should refer to the 2007 survey by R. Pemantle
\cite{pemantle_survey}. The process we are going to study in this paper is both a generalisation of P\'{o}lya's original process and a particular case of a model of interacting urns introduced by the author in \cite{IUM} in 2010 (see also Remark 1).

Fix $d$ in $\N$ the number of balls to be added to the urn at each step.

\begin{defi}
A sequence $(x_n)_{n\in\N\cup\{0\}}=(r_n,g_n)_{n\in\N\cup\{0\}}\in ((\N\cup\{0\})^2)^{\N\cup\{0\}}$ is called an urn path (with parameter $d$) if it satisfies:
\begin{itemize}
\item $x_0 = (r_0,g_0) = (0,0)$ ;
\item $r_n$ and $g_n$ are non-decreasing ;
\item $\forall n\in\N\cup\{0\}$, $r_n+g_n=dn$ or equivalently $(r_{n+1}+g_{n+1})-(r_n+g_n)=d$.
\end{itemize}
Here $r_n$ and $g_n$ should be interpreted respectively as the number of red and green balls in the urn.
\end{defi}

Here is an illustration of urn path with $d=5$ up to $n=3$:
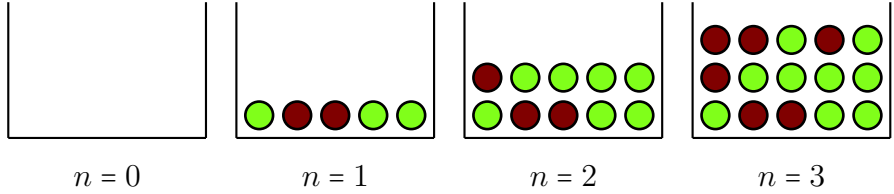
\begin{center}
	\begin{pspicture}(0,-.6)(11.6,1.8)
			\psline[border=0pt](0,1.8)(0,0)
			\psline[border=0pt](0,0)(2.6,0)
			\psline[border=0pt](2.6,0)(2.6,1.8)
			
			\psline[border=0pt](3.0,1.8)(3.0,0)
			\psline[border=0pt](3.0,0)(5.6,0)
			\psline[border=0pt](5.6,0)(5.6,1.8)

			\psline[border=0pt](6,1.8)(6,0)
			\psline[border=0pt](6,0)(8.6,0)
			\psline[border=0pt](8.6,0)(8.6,1.8)
			
			\psline[border=0pt](9,1.8)(9,0)
			\psline[border=0pt](9,0)(11.6,0)
			\psline[border=0pt](11.6,0)(11.6,1.8)

			\pscircle[fillstyle=solid,fillcolor=redb, linewidth=1pt,linecolor=black](3.3,.3){.2}
			\pscircle[fillstyle=solid,fillcolor=greenb, linewidth=1pt,linecolor=black](3.8,.3){.2}
			\pscircle[fillstyle=solid,fillcolor=greenb, linewidth=1pt,linecolor=black](4.3,.3){.2}
			\pscircle[fillstyle=solid,fillcolor=redb, linewidth=1pt,linecolor=black](4.8,.3){.2}
			\pscircle[fillstyle=solid,fillcolor=redb, linewidth=1pt,linecolor=black](5.3,.3){.2}
						
			\pscircle[fillstyle=solid,fillcolor=redb, linewidth=1pt,linecolor=black](6.3,.3){.2}
			\pscircle[fillstyle=solid,fillcolor=greenb, linewidth=1pt,linecolor=black](6.8,.3){.2}
			\pscircle[fillstyle=solid,fillcolor=greenb, linewidth=1pt,linecolor=black](7.3,.3){.2}
			\pscircle[fillstyle=solid,fillcolor=redb, linewidth=1pt,linecolor=black](7.8,.3){.2}
			\pscircle[fillstyle=solid,fillcolor=redb, linewidth=1pt,linecolor=black](8.3,.3){.2}

			\pscircle[fillstyle=solid,fillcolor=redb, linewidth=1pt,linecolor=black](9.3,.3){.2}
			\pscircle[fillstyle=solid,fillcolor=greenb, linewidth=1pt,linecolor=black](9.8,.3){.2}
			\pscircle[fillstyle=solid,fillcolor=greenb, linewidth=1pt,linecolor=black](10.3,.3){.2}
			\pscircle[fillstyle=solid,fillcolor=redb, linewidth=1pt,linecolor=black](10.8,.3){.2}
			\pscircle[fillstyle=solid,fillcolor=redb, linewidth=1pt,linecolor=black](11.3,.3){.2}

			\pscircle[fillstyle=solid,fillcolor=greenb, linewidth=1pt,linecolor=black](6.3,.8){.2}
			\pscircle[fillstyle=solid,fillcolor=redb, linewidth=1pt,linecolor=black](6.8,.8){.2}
			\pscircle[fillstyle=solid,fillcolor=redb, linewidth=1pt,linecolor=black](7.3,.8){.2}
			\pscircle[fillstyle=solid,fillcolor=redb, linewidth=1pt,linecolor=black](7.8,.8){.2}
			\pscircle[fillstyle=solid,fillcolor=redb, linewidth=1pt,linecolor=black](8.3,.8){.2}

			\pscircle[fillstyle=solid,fillcolor=greenb, linewidth=1pt,linecolor=black](9.3,.8){.2}
			\pscircle[fillstyle=solid,fillcolor=redb, linewidth=1pt,linecolor=black](9.8,.8){.2}
			\pscircle[fillstyle=solid,fillcolor=redb, linewidth=1pt,linecolor=black](10.3,.8){.2}
			\pscircle[fillstyle=solid,fillcolor=redb, linewidth=1pt,linecolor=black](10.8,.8){.2}
			\pscircle[fillstyle=solid,fillcolor=redb, linewidth=1pt,linecolor=black](11.3,.8){.2}

			\pscircle[fillstyle=solid,fillcolor=greenb, linewidth=1pt,linecolor=black](9.3,1.3){.2}
			\pscircle[fillstyle=solid,fillcolor=greenb, linewidth=1pt,linecolor=black](9.8,1.3){.2}
			\pscircle[fillstyle=solid,fillcolor=redb, linewidth=1pt,linecolor=black](10.3,1.3){.2}
			\pscircle[fillstyle=solid,fillcolor=greenb, linewidth=1pt,linecolor=black](10.8,1.3){.2}
			\pscircle[fillstyle=solid,fillcolor=redb, linewidth=1pt,linecolor=black](11.3,1.3){.2}

			\rput(1.3,-.5){{$n=0$}}
			\rput(4.3,-.5){{$n=1$}}
			\rput(7.3,-.5){{$n=2$}}
			\rput(10.3,-.5){{$n=3$}}

	\end{pspicture}
	\end{center}

Let us now endow the set of all urn paths with a particular dynamics thus defining a stochastic process. Fix $(w_i)_{i\in\N\cup\{0\}}\in(\R_+)^{\N\cup\{0\}}$ the {\em reinforcement weight sequence} and define the quantity $\pi(r,g):=\frac{w_r}{w_r+w_g}$ which we shall understand as the probability of drawing a red ball among $r$ red and $g$ green balls. Note that $\pi(r,g)+\pi(g,r)=1$. For $n\geq 1$, define the increment of red (resp.~green) balls at time $n$ by $\Delta r_n=r_n-r_{n-1}$ (resp. $\Delta g_n=g_n-g_{n-1}$).
Then the dynamics is defined as follows: $X=(R_n,G_n)_{n\in\N}$ is a Markov chain with $X_0=(0,0)$ and the transition law
\[
\left\{
\begin{array}{l}
\Prob\left[\left.\Delta R_{n+1}=a\,\right|X_n=(r,g)\right]\mathbb{1}_{\{X_n=(r,g)\}}={d \choose a}\pi(r,g)^a\pi(g,r)^{d-a}\mathbb{1}_{\{X_n=(r,g)\}},~\forall a, 0\leq a\leq d ;\\
\Delta G_{n+1}=d-\Delta R_{n+1}.
\end{array}
\right.
\]
It is clear that $X$ is an urn path with parameter $d$ a.s.
In words, $\Delta R_{n+1}$ (and $\Delta G_{n+1}$) follows a binomial distribution $B(d,\pi(r,g))$ conditionally on $X_n=(r,g)$: each of the $d$ balls added at time $n$ is independently red with probability $\pi(R_n,G_n)$ and green with probability $\pi(G_n,R_n)$.

For $n\in\N$, denote by $\mathcal{F}_n$ the $\sigma$-field generated by the $n$ first steps:
$$\mathcal{F}_n =\sigma\left(X_0,X_1,\dots,X_n\right).$$

This model can be linked to multiple particles {\em Reinforced Random Walks}. Consider the following star shaped graph:
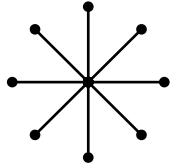
\begin{center}
	\begin{pspicture}(-1,-1)(1,1)
			\psline[linewidth=1pt]{*-*}(0,0)(1,0)
			\psline[linewidth=1pt]{*-*}(0,0)(-1,0)
			\psline[linewidth=1pt]{*-*}(0,0)(0,-1)
			\psline[linewidth=1pt]{*-*}(0,0)(0,1)
			\psline[linewidth=1pt]{*-*}(0,0)(.7,.7)
			\psline[linewidth=1pt]{*-*}(0,0)(.7,-.7)
			\psline[linewidth=1pt]{*-*}(0,0)(-.7,.7)
			\psline[linewidth=1pt]{*-*}(0,0)(-.7,-.7)
	\end{pspicture}
	\end{center}
Suppose that there are $d$ particles on the central vertex and at each step each particle jumps over one of the edges with probability proportional to $w_i$, where $i$ is the number of time the edge has been traversed by one of the particles since the begining, and then jumps back to the central vertex. This dynamic is equivalent to an urn process in which $d$ balls are added to the urn at each step and the balls could be of $E$ different colours where $E$ is the number of edges (or the number of leaf). In this paper we will limit ourselves to two different colours, that is $E=2$:
\begin{center}
	\begin{pspicture}(-1,0)(1,0)
			\psline[linewidth=1pt]{*-*}(0,0)(1,0)
			\psline[linewidth=1pt]{*-*}(0,0)(-1,0)
	\end{pspicture}
	\end{center}
	
\noindent \textbf{Remark 1.} This model is equivalent to the {\em Interacting Urn Model} \cite{IUM} with $d$ urns in the case when the memory sharing is maximal, that is the correlation probability $p=1$. In that setting all the $d$ urns always draw their balls in the $d$ urns combined. Therefore, Theorem \ref{theoreme} answers an open question of \cite{IUM}.

For the proofs it will be convenient to use a finer time indexing in which only one ball is added at each step. For this time indexing we shall reserve the superscript $k$ and an additional tilde: denote by $(\tilde{X}_k)_{k\in\N\cup\{0\}}=(\tilde{R}_k,\tilde{G}_k)_{k\in\N\cup\{0\}}$ a random urn path with parameter $d=\tilde{d}=1$ defined by
\[
\left\{
\begin{array}{l}
\begin{split}
\Prob\left[\left.\Delta \tilde{R}_{k+1}=\delta\right|(\tilde{X}_\ell)_{\ell\leq k}=(\tilde{x}_\ell)_{\ell\leq k}\right]=\pi(\tilde{r}_{\lfloor k+1\rfloor_d},\tilde{g}_{\lfloor k+1\rfloor_d})^\delta\pi(\tilde{g}_{\lfloor k+1\rfloor_d},\tilde{r}_{\lfloor k+1\rfloor_d})^{1-\delta},\\
\text{on }\{(\tilde{X}_\ell)_{\ell\leq k}=(\tilde{x}_\ell)_{\ell\leq k}\}\text{ for }\delta = 0\text{ or }1 ;
\end{split}\\
\Delta \tilde{G}_{k+1}=1-\Delta \tilde{R}_{k+1},
\end{array}
\right.
\]
for $k\geq 1$ where
$$\lfloor k \rfloor_d = \max\{\ell \in d\Z :\ell< k\}=d\left\lfloor\frac{k-1}{d}\right\rfloor.$$

In words, at each time $k$ we add a ball that is drawn in the urn using the configuration at time $\lfloor k \rfloor_d$, which is the last time when the number of balls in the urn was a multiple of $d$. Note that $\lfloor k \rfloor_1 = k-1$ (and not $k$), we prefer to use this notation since it matches well the dynamics described above.

Then the process $X=(\tilde{R}_{dn},\tilde{G}_{dn})_{n\in\N}$ follows the law of the urn process with parameter $d$ previously defined. The proof of this fact is left to the reader.

\section{First results}
An interesting hypothesis about the reinforcement weight sequence is the {\em Strong Reinforcement Hypothesis}:
\begin{equation}
\label{SRH}
\tag{SRH}
s_{\infty}:=\sum_{i=0}^{\infty}\frac{1}{w_i}<\infty.
\end{equation}
The reason why this hypothesis is natural to consider is the following result of Davis \cite{davis}. Denote by $A$ the event
\begin{equation}
A:=\left\{
\begin{split}
\text{There exists a time } n \text{ such that after } n \text{ all the balls}\\ \text{drawn out of the urn have the same colour}
\end{split}
\right\}.
\end{equation}
Then we have the following proposition.
\begin{prop}(Davis, \cite{davis})
If $d=1$, then $\mathbb{P}\left[A\right]=1$ if (\ref{SRH}) is satisfied and $\mathbb{P}\left[A\right]=0$ otherwise.
\end{prop}

An elegant proof of this proposition using a continuous time construction of urn processes can be found in \cite{davis} or \cite{pemantle_survey}.

It is therefore natural to wonder whether such a result is true when $d\geq 2$. It is believable that the following conjecture is true but difficult to prove.

\begin{conj}
\label{conj}
For $d\geq 2$, if (\ref{SRH}) is satisfied then $\mathbb{P}\left[A\right]=1$.
\end{conj}

The main result of this paper is the weaker result:

\begin{theorem}
\label{theoreme}
For $d\geq 2$,  if (\ref{SRH}) is satisfied and $w$ is non-decreasing then $\mathbb{P}\left[A\right]=1$.
\end{theorem}

The proof of this theorem will be given in the next section. It is actually an adaptation of an argument given by Limic and Tarr\`es in \cite{limictarres} showing that a reinforced random walk on some general graph is eventually attracted by a single edge, assuming that the reinforcement sequence satisfies the (\ref{SRH}) and additional technical assumptions including the non-decreasing case.

Note that for $d\geq 2$ the reciprocal of Conjecture \ref{conj} is not true since we have the following proposition.

\begin{prop}
\label{lemme1}
If $d\geq 2$, there exist a reinforcement weight sequence which does not satisfy (\ref{SRH}) but such that $A$ happens almost surely.
\end{prop}

\begin{proof}
We construct a counterexample. Fix $\rho>1$ and define for instance $(w_i)_{i\in\N\cup\{0\}}$ as follows:
\[
  w_i = \left\{
          \begin{array}{ll}
            1 & \qquad \text{if } i\text{ is a multiple of }d; \\
            \rho^i & \qquad \text{else.}\\
          \end{array}
        \right.
\]
Obviously this $w$ does not satisfy (\ref{SRH}) since there are infinitely many $i$ such that $w_i=1$. The proof proceeds in two steps:
\begin{enumerate}[i)]
\item There exist almost surely infinitely many $n\geq 1$ such that both $R_n$ and $G_n$ are not multiple of $d$.
\item For each time $n$ such that $R_n$ and $G_n$ are not multiple of $d$, suppose without loss of generality that $R_n\geq G_n$. There is a positive probability uniformly bounded away from $0$ that the urn draws only red balls after time $n$.
\end{enumerate}
Lemma \ref{lemme1} follows from those two points by L\'{e}vy 0-1 law.

\noindent \textbf{First step.} For each $n$ such that $R_n$ is a multiple of $d$, the probability for $R_{n+1}$ to be also a multiple of $d$ equals the probability that all the balls draw at time $n+1$ have the same colour. It equals $2\times (1/2)^d=2^{1-d}<1$ (there is $2$ choices for the colour and then each of the $d$ balls is of this colour with probability $\pi(R_n,G_n)=\pi(G_n,R_n)=\frac{1}{1+1}=1/2$ since $w_i=1$ when $i$ is a multiple of $d$). Therefore after geometric one half many steps $R_m$ is not anymore a multiple of $d$. This proves the first point.

\noindent \textbf{Second step.} Suppose that $R_n$ is not a multiple of $d$. We use classical expectation calculations to find a lower bound uniform in $n$ for the probability to draw always the same colour after time $n$. Let us suppose without loss of generality that $R_n\geq G_n$, so that on $\{R_n\not\in d\N\}\cup\{R_n\geq G_n\}$ the probability to draw only red balls after time $n$ is
\begin{eqnarray*}
\Prob\left[\left.R_{n+k}=R_{n+k-1}+d,\forall k\geq 1\right|\mathcal{F}_n\right]&=&\E\left[\left.\prod_{k=1}^\infty\mathbb{1}_{\left\{R_{n+k}=R_{n+k-1}+d\right\}}\right|\mathcal{F}_n\right];\\
&=&\lim_{K\to\infty}\E\left[\left.\prod_{k=1}^K\mathbb{1}_{\left\{R_{n+k}=R_{n+k-1}+d\right\}}\right|\mathcal{F}_n\right];
\end{eqnarray*}

by monotone convergence theorem. An then
\begin{eqnarray*}
\Prob\left[\left.R_{n+k}=R_{n+k-1}+d,\forall k\geq 1\right|\mathcal{F}_n\right]&=&\lim_{K\to\infty}\E\left[\left.\E\left[\left.\mathbb{1}_{\left\{R_{n+K}=R_{n+K-1}+d\right\}}\right|\mathcal{F}_{n+K-1}\right]\prod_{k=1}^{K-1}\mathbb{1}_{\left\{R_{n+k}=R_{n+k-1}+d\right\}}\right|\mathcal{F}_n\right];\\
&=&\lim_{K\to\infty}\E\left[\left.\left(\frac{\rho^{R_{n+K-1}}}{\rho^{R_{n+K-1}}+\rho^{G_{n+K-1}}}\right)^d\prod_{k=1}^{K-1}\mathbb{1}_{\left\{R_{n+k}=R_{n+k-1}+d\right\}}\right|\mathcal{F}_n\right];\\
&=&\lim_{K\to\infty}\E\left[\left.\left(\frac{1}{1+\rho^{G_{n}-R_{n}-(K-1)d}}\right)^d\prod_{k=1}^{K-1}\mathbb{1}_{\left\{R_{n+k}=R_{n+k-1}+d\right\}}\right|\mathcal{F}_n\right];\\
&\geq&\lim_{K\to\infty}\E\left[\left.\left(\frac{1}{1+\rho^{-(K-1)d}}\right)^d\prod_{k=1}^{K-1}\mathbb{1}_{\left\{R_{n+k}=R_{n+k-1}+d\right\}}\right|\mathcal{F}_n\right];\\
\end{eqnarray*}
where we use the fact that, on $\bigcap_{k=1}^{K-1}\left\{R_{n+k}=R_{n+k-1}+d\right\}$, $G_{n+K-1}-R_{n+K-1}=G_n-R_n-(K-1)d$ and where the last inequality is due to $R_n\geq G_n$. So by induction
\begin{eqnarray*}
\Prob\left[\left.R_{n+k}=R_{n+k-1}+d,\forall k\geq 1\right|\mathcal{F}_n\right]&\geq&\lim_{K\to\infty}\left(\frac{1}{1+\rho^{-(K-1)d}}\right)^d\E\left[\left.\prod_{k=1}^{K-1}\mathbb{1}_{\left\{R_{n+k}=R_{n+k-1}+d\right\}}\right|\mathcal{F}_n\right];\\
&\geq&\lim_{K\to\infty}\prod_{k=1}^{K}\left(\frac{1}{1+\rho^{-(k-1)d}}\right)^d=\left(\prod_{k=1}^{\infty}\frac{1}{1+\rho^{-(k-1)d}}\right)^d>0.\\
\end{eqnarray*}
This lower bound is positive and uniform in $n$, so it proves the proposition.
\end{proof}

\section{Proof of Theorem \ref{theoreme}}

To simplify notations, denote by $\mathfrak{c}_k$ the colour that is drawn in the urn at time $k$:
$$\mathfrak{c}_k=\left\{\begin{array}{ll}
\mathfrak{r}&\qquad\text{if}\quad\tilde{R}_{k}=\tilde{R}_{k-1}+1;\\
\mathfrak{g}&\qquad\text{if}\quad\tilde{G}_{k}=\tilde{G}_{k-1}+1.\\
\end{array}\right.$$
Then, define the following process:

$$N_k:=\sum_{\ell=1}^k \left( \frac{\mathbb{1}_{\{\mathfrak{c}_\ell=\mathfrak{r}\}}}{w_{\tilde{R}_{l-1}}}-\frac{\mathbb{1}_{\{\mathfrak{c}_\ell=\mathfrak{g}\}}}{w_{\tilde{G}_{l-1}}}\right)=\sum_{i=1}^{\tilde{R}_{k}}\frac{1}{w_{i-1}}-\sum_{i=1}^{\tilde{G}_{k}}\frac{1}{w_{i-1}},$$

The two reasons for which the process $N$ is interesting are on the one hand its relation to the event $A$:
$$A^c=\left\{R_k\underset{k\to\infty}{\longrightarrow}\infty\text{ and }G_k\underset{k\to\infty}{\longrightarrow}\infty\right\}=\left\{N_k\underset{k\to\infty}{\longrightarrow}0\right\}.$$
\noindent and on the other hand the fact that it is closely related to the other process
$$M_k:=\sum_{\ell=1}^k \left( \frac{\mathbb{1}_{\{\mathfrak{c}_\ell=\mathfrak{r}\}}}{w_{\tilde{R}_{\lfloor l\rfloor_d}}}-\frac{\mathbb{1}_{\{\mathfrak{c}_\ell=\mathfrak{g}\}}}{w_{\tilde{G}_{\lfloor l\rfloor_d}}}\right),$$
\noindent which is a martingle because
\begin{eqnarray*}
\mathbb{E}\left[\left.M_{k}-M_{k-1}\right|\mathcal{F}_{k-1}\right]&=&\frac{1}{w_{\tilde{R}_{\lfloor k\rfloor_d}}}\mathbb{P}\left[\left.\mathfrak{c}_k=\mathfrak{r}\right|\mathcal{F}_{k-1}\right]-\frac{1}{w_{\tilde{G}_{\lfloor k\rfloor_d}}}\mathbb{P}\left[\left.\mathfrak{c}_k=\mathfrak{g}\right|\mathcal{F}_{k-1}\right]\\
&=&\frac{1}{w_{\tilde{R}_{\lfloor k\rfloor_d}}}\frac{w_{\tilde{R}_{\lfloor k\rfloor_d}}}{w_{\tilde{R}_{\lfloor k\rfloor_d}}+w_{\tilde{G}_{\lfloor k\rfloor_d}}}-\frac{1}{w_{\tilde{G}_{\lfloor k\rfloor_d}}}\frac{w_{\tilde{G}_{\lfloor k\rfloor_d}}}{w_{\tilde{R}_{\lfloor k\rfloor_d}}+w_{\tilde{G}_{\lfloor k\rfloor_d}}}=0,
\end{eqnarray*}
and will be very helpful in the proof. Note that the two quantities

$$\lim_{k\to\infty}N_k\text{ and }\lim_{k\to\infty}M_k$$

\noindent exist almost surely and are finite since both $N$ and $M$ are defined as differences of two non-decreasing sequences with a finite upper bound. We denote these two limits by $N_\infty$ and $M_\infty$ respectively. If $d=1$ then $M=N$ and in that case one can check that the assumption that $w$ is non-decreasing is not needed for this argument to work.

Thus, to prove Theorem \ref{theoreme}, we have to prove that $N_\infty\neq 0$ almost surely.

Let us first give a rough sketch of the proof. We define
$$X_k:=\min(R_{\lfloor k \rfloor_d},G_{\lfloor k \rfloor_d})~~~\text{   and   }~~~B_k:=\sum_{i=X_k}^{\infty}\frac{1}{w_i^2},$$
\noindent which will be an interesting quantity to compare the variance of $N_\infty-N_k$ with. A rough sketch of the proof of Theorem \ref{theoreme} is given by the following picture.

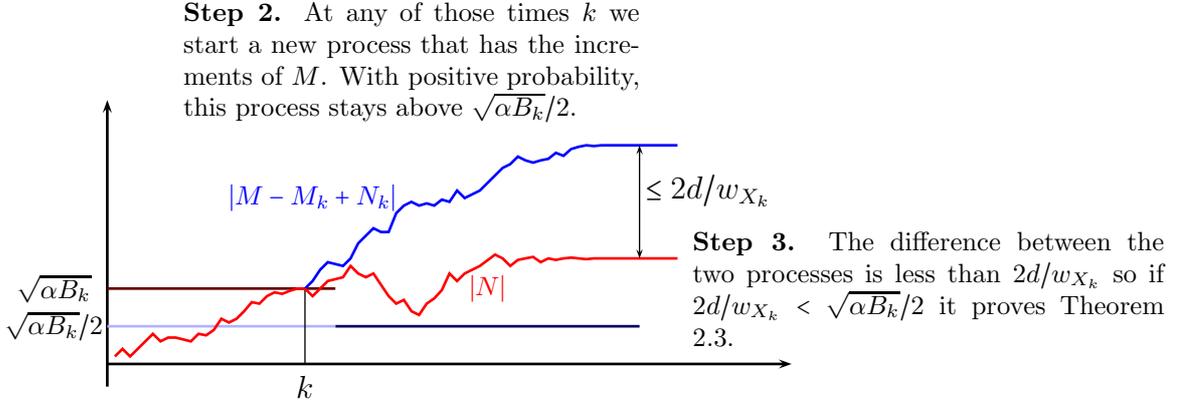
\begin{center}
	\begin{pspicture}(-1.5,-2.7)(13.6,5)
			\psline[linewidth=1pt]{->}(0,0)(9,0)
			\psline[linewidth=1pt]{->}(0,-.3)(0,3.5)
			\psline[linewidth=1pt, linecolor=darkred]{-}(0,1)(3,1)
			
			\psline[linewidth=1pt, linecolor=darkblue]{-}(3,.5)(7,.5)
			\psline[linewidth=1pt, linecolor=lightblue]{-}(0,.5)(3,.5)
			
			\psline[linewidth=.5pt, linecolor=black]{-}(2.6,0)(2.6,1)
			
			\psline[linewidth=1pt, linecolor=red]{-}(0.1,0.1)(0.2,0.2)(0.3,0.1)(0.4,0.2)(0.5,0.3)(0.6,0.4)(0.7,0.3)(0.8,0.35)(0.9,0.35)(1.1,0.3)(1.2,0.4)(1.3,0.38)(1.4,0.45)(1.5,0.6)(1.6,0.55)(1.7,0.6)(1.8,0.7)(1.9,0.82)(2,0.8)(2.1,0.9)(2.2,0.95)(2.3,0.93)(2.4,0.97)(2.5,1)(2.6,1)
			\psline[linewidth=1pt, linecolor=blue]{-}(2.6,1)(2.7,1.1)(2.8,1.25)(2.9,1.35)(3.1,1.3)(3.2,1.4)(3.3,1.6)(3.4,1.7)(3.5,1.8)(3.6,1.75)(3.7,1.75)(3.8,2)(3.9,2.1)(4,2.15)(4.1,2.1)(4.2,2.13)(4.3,2.1)(4.4,2.18)(4.5,2.15)(4.6,2.3)(4.7,2.2)(4.8,2.25)(4.9,2.3)(5.1,2.5)(5.2,2.6)(5.3,2.65)(5.4,2.75)(5.5,2.7)(5.6,2.67)(5.7,2.7)(5.8,2.72)(5.9,2.8)(6,2.76)(6.1,2.85)(6.2,2.88)(6.3,2.9)(6.4,2.89)(6.5,2.9)(6.6,2.9)(6.7,2.9)(6.8,2.9)(6.9,2.9)(7,2.9)(7.5,2.9)
			\psline[linewidth=1pt, linecolor=red]{-}(2.6,1)(2.7,.9)(2.8,1)(2.9,1.1)(3.1,1.15)(3.2,1.3)(3.3,1.2)(3.4,1.15)(3.5,1.2)(3.6,1.05)(3.7,.9)(3.8,.8)(3.9,.85)(4,.7)(4.1,.65)(4.2,.8)(4.3,.87)(4.4,.98)(4.5,1.2)(4.6,1.1)(4.7,1.2)(4.8,1.25)(4.9,1.3)(5.1,1.45)(5.2,1.4)(5.3,1.3)(5.4,1.38)(5.5,1.4)(5.6,1.42)(5.7,1.35)(5.8,1.4)(5.9,1.38)(6,1.4)(6.1,1.41)(6.2,1.4)(6.3,1.39)(6.4,1.4)(6.5,1.4)(6.6,1.4)(6.7,1.4)(6.8,1.4)(6.9,1.4)(7,1.4)(7.5,1.4)
			
			\psline[linewidth=.5pt, linecolor=black]{<->}(7,2.9)(7,1.4)

			\rput(5,1){\footnotesize\red $|N|$}
			\rput(2.7,2.2){\footnotesize\blue $|M-M_k+N_k|$}
			
			\rput(2.6,-.3){$k$}
			\rput(7.9,2.3){$\leq 2d/w_{X_k}$}
			
			\rput(-.7,.5){\footnotesize$\sqrt{\alpha B_k}/2$}
			\rput(-.7,1){\footnotesize$\sqrt{\alpha B_k}$}
			
			\rput(1.5,-1.3){\parbox[c]{6cm}{\footnotesize\textbf{Step 1.} Either $N$ does not go to 0 or there are infinitely many times $k$ such that $|N_k|\geq \sqrt{\alpha B_k}$.}}
			
			\rput(4,4){\parbox[c]{6cm}{\footnotesize\textbf{Step 2.} At any of those times $k$ we start a new process that has the increments of $M$. With positive probability, this process stays above $\sqrt{\alpha B_k}/2$.}}

			\rput(10.8,1){\parbox[c]{6.2cm}{\footnotesize\textbf{Step 3.} The difference between the two processes is less than $2d/w_{X_k}$ so if $2d/w_{X_k}<\sqrt{\alpha B_k}/2$ it proves Theorem \ref{theoreme}.}}
			
			\rput(9.5,-1.7){\parbox[c]{6.2cm}{\footnotesize\textbf{Step 4.} If $2d/w_{X_k}\geq\sqrt{\alpha B_k}/2$ it means that the reinforcement is very strong and there is a much simpler direct proof of Theorem \ref{theoreme}.}}
			
	\end{pspicture}
	\end{center}
	
For the proof to work, we have to choose $\alpha$ for instance equal to $(24+16d s_\infty)^{-2}$. We will see in the proof how this quantity appears.

Let us start with the step 1 that will be given by Lemma \ref{lemme2}. For all $k\geq 0$ we define:
$$T_k:=\min\left\{\ell\geq k: N_\ell^2\geq\alpha B_k \right\},$$
\noindent and
$$\mathcal{J}:=\left\{T_k<\infty\right\}\cup\left\{N_\infty\neq0\right\}=\left\{N_\infty^{T_k}\neq0\right\}.$$

\begin{lemme}\label{lemme2} We have
$$\mathbb{P}\left[\left.\mathcal{J}\right|\mathcal{F}_k\right]\geq \frac{1}{12}>0.$$
In words it means that with a positive probability uniformly bounded away from 0 (the value 1/12 has actually no importance), the process $N$ either does not go to 0, or exits the interval $[-\sqrt{\alpha B_k},\sqrt{\alpha B_k}]$.
\end{lemme}

\begin{proof} For some convenient reason we will focus our analysis on the stopped process $N^{T_k}$ instead of $N$.
To prove this lemma we will use two calculations using the variance of $N^{T_k}$. The first one is as follows:
\begin{eqnarray*}
B_k\times\mathbb{P}\left[\left.\mathcal{J}^c\right|\mathcal{F}_k\right]&=&\mathbb{E}\left[\left.\mathbb{1}_{\mathcal{J}^c}B_k\right|\mathcal{F}_k\right]\\
&=&\mathbb{E}\left[\left.\mathbb{1}_{\mathcal{J}^c}\sum_{i=X_k}^{\infty}\frac{1}{w_i^2}\right|\mathcal{F}_k\right]\\
&\leq&\mathbb{E}\left[\left.\mathbb{1}_{\mathcal{J}^c}\sum_{i=k}^{\infty}\left(N_{i+1}^{T_k}-N_i^{T_k}\right)^2\right|\mathcal{F}_k\right],\\
\end{eqnarray*}
because on $\mathcal{J}^c$, we have $T_k=\infty$ so any $1/w_i^2$ from $X_k$ to $\infty$ appears at leas once in the last sum. Then
\begin{eqnarray*}
B_k\times\mathbb{P}\left[\left.\mathcal{J}^c\right|\mathcal{F}_k\right]&\leq&\mathbb{E}\left[\left.\sum_{i=k}^{\infty}\left(N_{i+1}^{T_k}-N_i^{T_k}\right)^2\right|\mathcal{F}_k\right]\\
&=&\mathbb{E}\left[\left.\sum_{i=k}^{\infty}\left((N_{i+1}^{T_k})^2-(N_i^{T_k})^2-2N_i^{T_k}\left(N_{i+1}^{T_k}-N_i^{T_k}\right)\right)\right|\mathcal{F}_k\right]\\
&=&\mathbb{E}\left[\left.\sum_{i=k}^{\infty}(N_{i+1}^{T_k})^2-(N_i^{T_k})^2\right|\mathcal{F}_k\right]-2\mathbb{E}\left[\left.\sum_{i=k}^{\infty}N_i^{T_k}\left(N_{i+1}^{T_k}-N_i^{T_k}\right)\right|\mathcal{F}_k\right]\\
&\leq&\mathbb{E}\left[\left.\sum_{i=k}^{\infty}(N_{i+1}^{T_k})^2-(N_i^{T_k})^2\right|\mathcal{F}_k\right]+4d s_\infty \alpha B_k,\\
\end{eqnarray*}
where the last inequality is due to the upper bound:
\begin{eqnarray*}
\left|\mathbb{E}\left[\left.\sum_{i=k}^{\infty}N_i^{T_k}\left(N_{i+1}^{T_k}-N_i^{T_k}\right)\right|\mathcal{F}_k\right]\right|&\leq&\mathbb{E}\left[\left.\sum_{i=k}^{\infty}\left|N_i^{T_k}\right|\left|N_{i+1}^{T_k}-N_i^{T_k}\right|\right|\mathcal{F}_k\right]\\
&\leq& \alpha B_k\mathbb{E}\left[\left.\sum_{i=k}^{\infty}\left|N_{i+1}^{T_k}-N_i^{T_k}\right|\right|\mathcal{F}_k\right]\leq 2ds_{\infty}\alpha B_k,\\
\end{eqnarray*}
because in the last sum, each term of the form $1/w_k$ appears at most $2d$ times ($d$ times for drawings of red balls and $d$ times for drawings of green balls). Then :

\begin{eqnarray*}
B_k\times\mathbb{P}\left[\left.\mathcal{J}^c\right|\mathcal{F}_k\right]&\leq&\mathbb{E}\left[\left.\sum_{i=k}^{\infty}(N_{i+1}^{T_k})^2-(N_i^{T_k})^2\right|\mathcal{F}_k\right]+4d s_\infty \alpha B_k\\
&\leq&\mathbb{E}\left[\left.(N_\infty^{T_k})^2-(N_k^{T_k})^2\right|\mathcal{F}_k\right]+4d s_\infty \alpha B_k\\
&\leq&\mathbb{E}\left[\left.\left(N_\infty^{T_k}-N_k^{T_k}\right)^2\right|\mathcal{F}_k\right]+2N_k^{T_k}\mathbb{E}\left[\left.\left(N_\infty^{T_k}-N_k^{T_k}\right)\right|\mathcal{F}_k\right]+4d s_\infty \alpha B_k\\
&\leq&\mathbb{E}\left[\left.\left(N_\infty^{T_k}-N_k^{T_k}\right)^2\right|\mathcal{F}_k\right]+2\sqrt{\alpha B_k}\left(2\sqrt{\alpha B_k}+\sqrt{B_k}\right)+4d s_\infty \alpha B_k\\
&\leq&(N_k^{T_k})^2\mathbb{P}\left[\left.\mathcal{J}^c\right|\mathcal{F}_k\right]+\mathbb{E}\left[\left.\left(N_\infty^{T_k}-N_k^{T_k}\right)^2\mathbb{1}_{\mathcal{J}}\right|\mathcal{F}_k\right]+B_k(4\alpha+2\sqrt{\alpha}+4d\alpha s_\infty)\\
&\leq&(N_k^{T_k})^2\mathbb{P}\left[\left.\mathcal{J}^c\right|\mathcal{F}_k\right]+\mathbb{P}\left[\left.\mathcal{J}\right|\mathcal{F}_k\right]\left(2\left(\sqrt{\alpha B_k}+\sqrt{B_k}\right)^2+2(N_k^{T_k})^2\right)\\
&&+B_k(4\alpha+2\sqrt{\alpha}+4d\alpha s_\infty).\\
\end{eqnarray*}
So we have:
$$\mathbb{P}\left[\left.\mathcal{J}^c\right|\mathcal{F}_k\right]\left(B_k-N_k^2\right)\leq\mathbb{P}\left[\left.\mathcal{J}\right|\mathcal{F}_k\right]\left(2\alpha B_k+4\sqrt{\alpha} B_k + B_k+2N_k^2\right)+B_k(4\alpha+2\sqrt{\alpha}+4d\alpha s_\infty)$$

\noindent Using $\mathbb{P}\left[\left.\mathcal{J}^c\right|\mathcal{F}_k\right]=1-\mathbb{P}\left[\left.\mathcal{J}\right|\mathcal{F}_k\right]$ this could be rewritten as:

$$B_k-N_k^2\leq\mathbb{P}\left[\left.\mathcal{J}\right|\mathcal{F}_k\right]\left(2\alpha B_k+4\sqrt{\alpha} B_k + 3B_k + N_k^2\right)+B_k(4\alpha+2\sqrt{\alpha}+4d\alpha s_\infty).$$

And so, on $\{N_k^2\leq B_k/2\}$ we have:

$$\mathbb{P}\left[\left.\mathcal{J}\right|\mathcal{F}_k\right]\geq\frac{\frac{1}{2}-4\alpha-2\sqrt{\alpha}-4d\alpha s_\infty}{2\alpha+4\sqrt{\alpha}+3+\frac{1}{2}}.$$
Any $\alpha$ small enough so that the last quantity is in $(0,1)$ would suffice here. For instance with the choice $\alpha=(24+16d s_\infty)^{-2}$ we obtain $\mathbb{P}\left[\left.\mathcal{J}\right|\mathcal{F}_k\right]\geq1/12$ and Lemma \ref{lemme2}.
\end{proof}

By L\'evy 0-1 law, Lemma \ref{lemme2} gives that almost surely either $N_\infty\neq 0$ or there are infinitely many times $k_0$ such that $|N_{k_0}|\geq \sqrt{\alpha B_{k_0}}$. It is now time to compare $N$ with the martingale $M$. More precisely for any $k_0$ we introduce the process $M^{k_0}_k=M_k-M_{k_0}+N_{k_0}$ for $k\geq {k_0}$ which is equal to $N$ at time $k={k_0}$ and has the increments of $M$ after time $k_0$.

Let us start with an upper bound for the difference between the increments of $N$ and $M$:
\begin{eqnarray*}
\left|M^{k_0}_{k_0+k}-N_{k_0+k}\right|&=&\left|\left(M_{k_0+k}-M_{k_0}\right)-\left(N_{k_0+k}-N_{k_0}\right)\right|\\
&=&\left|\sum_{\ell=k_0+1}^{k_0+k}\left( \frac{\mathbb{1}_{\{\mathfrak{c}_\ell=\mathfrak{r}\}}}{w_{\tilde{R}_{\lfloor l\rfloor_d}}}-\frac{\mathbb{1}_{\{\mathfrak{c}_\ell=\mathfrak{g}\}}}{w_{\tilde{G}_{\lfloor l\rfloor_d}}}\right)-\sum_{\ell={k_0}+1}^{{k_0}+k}\left( \frac{\mathbb{1}_{\{\mathfrak{c}_\ell=\mathfrak{r}\}}}{w_{\tilde{R}_{l-1}}}-\frac{\mathbb{1}_{\{\mathfrak{c}_\ell=\mathfrak{g}\}}}{w_{\tilde{G}_{l-1}}}\right)\right|\\
&\leq&\left|\sum_{\ell={k_0}+1}^{{k_0}+k} \frac{\mathbb{1}_{\{\mathfrak{c}_\ell=\mathfrak{r}\}}}{w_{\tilde{R}_{\lfloor l\rfloor_d}}}-\frac{\mathbb{1}_{\{\mathfrak{c}_\ell=\mathfrak{r}\}}}{w_{\tilde{R}_{l-1}}}\right|+\left|\sum_{\ell={k_0}+1}^{{k_0}+k} \frac{\mathbb{1}_{\{\mathfrak{c}_\ell=\mathfrak{g}\}}}{w_{\tilde{G}_{\lfloor l\rfloor_d}}}-\frac{\mathbb{1}_{\{\mathfrak{c}_\ell=\mathfrak{g}\}}}{w_{\tilde{G}_{l-1}}}\right|\\
\end{eqnarray*}

Let us find an upper bound for the first term of this sum. The same upper bound will hold for the second term. First notice that since $w$ is non-decreasing this is a sum of positive numbers so the absolute value is not needed. Then

\begin{eqnarray*}
\sum_{\ell={k_0}+1}^{{k_0}+k} \frac{\mathbb{1}_{\{\mathfrak{c}_\ell=\mathfrak{r}\}}}{w_{\tilde{R}_{\lfloor l\rfloor_d}}}-\frac{\mathbb{1}_{\{\mathfrak{c}_\ell=\mathfrak{r}\}}}{w_{\tilde{R}_{l-1}}}&\leq&\sum_{\ell=\lfloor {k_0}+1\rfloor_d\,+1}^{\infty} \frac{\mathbb{1}_{\{\mathfrak{c}_\ell=\mathfrak{r}\}}}{w_{\tilde{R}_{\lfloor l\rfloor_d}}}-\frac{\mathbb{1}_{\{\mathfrak{c}_\ell=\mathfrak{r}\}}}{w_{\tilde{R}_{l-1}}}\\
&\leq&\sum_{\ell=\lfloor {k_0}+1\rfloor_d\,+1}^{\infty} \frac{1}{w_{\tilde{R}_{\lfloor \ell\rfloor_d}}}-\frac{1}{w_{\tilde{R}_{\ell-1}}}\\
&=&\sum_{\ell=\frac{\lfloor {k_0}+1\rfloor_d}{d}}^{\infty}\sum_{i=1}^{d} \frac{1}{w_{\tilde{R}_{\lfloor d\ell+i\rfloor_d}}}-\frac{1}{w_{\tilde{R}_{d\ell+i-1}}}\\
&=&\sum_{\ell=\frac{\lfloor {k_0}+1\rfloor_d}{d}}^{\infty}\sum_{i=1}^{d} \frac{1}{w_{\tilde{R}_{d\ell}}}-\frac{1}{w_{\tilde{R}_{d\ell+i-1}}}\\
&=&\sum_{i=1}^{d}\sum_{\ell=\frac{\lfloor {k_0}+1\rfloor_d}{d}}^{\infty} \frac{1}{w_{\tilde{R}_{d\ell}}}-\frac{1}{w_{\tilde{R}_{d\ell+i-1}}}\\
\end{eqnarray*}

The last identity is true because the sum is absolutely convergent. Now for each $i=1\ldots d$, the sum

$$\sum_{\ell=\frac{\lfloor k_0+1\rfloor_d}{d}}^{\infty} \frac{1}{w_{\tilde{R}_{dl}}}-\frac{1}{w_{\tilde{R}_{dl+i-1}}}$$
\noindent is alternate with the absolute values of its terms non-increasing because for each $\ell$:
$$\frac{1}{w_{\tilde{R}_{d\ell}}}\geq \frac{1}{w_{\tilde{R}_{d\ell+i-1}}}\geq\frac{1}{w_{\tilde{R}_{d(\ell+1)}}}$$
\noindent since $w$ is non-decreasing. We conclude that
$$0\leq \sum_{\ell=\frac{\lfloor k_0+1\rfloor_d}{d}}^{\infty} \frac{1}{w_{\tilde{R}_{dl}}}-\frac{1}{w_{\tilde{R}_{dl+i-1}}}\leq \frac{1}{w_{\tilde{R}_{\lfloor n+1\rfloor_d}}}$$

Then

$$\sum_{\ell=k_0+1}^{k_0+k} \frac{\mathbb{1}_{\{\mathfrak{c}_\ell=\mathfrak{r}\}}}{w_{\tilde{R}_{\lfloor l\rfloor_d}}}-\frac{\mathbb{1}_{\{\mathfrak{c}_\ell=\mathfrak{r}\}}}{w_{\tilde{R}_{l-1}}}\leq \frac{d}{w_{\tilde{R}_{\lfloor n+1\rfloor_d}}}$$

and

$$\left|M^{k_0}_{k_0+k}-N_{k_0+k}\right|\leq \frac{d}{w_{\tilde{R}_{\lfloor n+1\rfloor_d}}}+\frac{d}{w_{\tilde{G}_{\lfloor n+1\rfloor_d}}}\leq \frac{2d}{w_{\min(\tilde{R}_{\lfloor n+1\rfloor_d},\tilde{G}_{\lfloor n+1\rfloor_d})}}= \frac{2d}{w_{X_{n+1}}}\leq \frac{2d}{w_{X_{n}}}.$$

Note that this upper bound does not depend on $k$ so we can write:

$$\sup_{k\leq 0}\left|M^{k_0}_{k_0+k}-N_{k_0+k}\right|\leq \frac{2d}{w_{X_n}}.$$

Given the result of Lemma \ref{lemme2} we now wish to compare this upper bound with $\sqrt{B_n}$. The proof will now bifurcate in two cases:

\textbf{Case 1.} If $\liminf_{k\to\infty}\frac{1}{w_{X_k}^2 B_k}=0$ then $N$ is very close to $M$ so it will suffice to prove that $M$ stays far away from 0.

\textbf{Case 2.} If $\liminf_{k\to\infty}\frac{1}{w_{X_k}^2 B_k}>0$ then the reinforcement is very strong and we will be able to give a direct proof that the urn always draws the same colour after a finite time.

Let us start with the first case. Since $\liminf_{k\to\infty}\frac{1}{w_{X_k}^2 B_k}=0$ then there exist infinitely many $k$ such that

$$\frac{2d}{w_{X_k}}< \frac{\sqrt{\alpha B_k}}{2}.$$

Note that these $k$ are deterministic and Lemma \ref{lemme2} holds for each of them. So almost surely either $N$ does not go to 0 or there is infinitely many of these $k$ for which there exist a time $T(k)>k$ such that $\left|N_{T(k)}\right|\geq\alpha \sqrt{B_k}$. In the latter case let us consider one such $k$ and start at time $T(k)$ the martingale
$$M^{T(k)}_\ell = M_\ell - M_{T(k)}+N_{T(k)}, ~~\forall \ell\ge T(k).$$

Then introduce the stopping timered

$$S=\min\left\{\ell\geq T(k): M^{T(k)}_\ell\not\in\left[\frac{\sqrt{\alpha B_k}}{2},3d\sqrt{B_k}\right] \right\}.$$

We have :
$$\sqrt{\alpha B_k}\leq M^{T(k)}_{T(k)}=\mathbb{E}\left[\right.M^{T(k)}_{S}\left|\mathcal{F}_{T_k}\right]\leq \frac{\sqrt{\alpha B_k}}{2}\mathbb{P}\left[M^{T(k)}_{S}<\frac{\sqrt{\alpha B_k}}{2}\right]+(3d+1)\sqrt{B_k}\mathbb{P}\left[M^{T(k)}_{S}\geq\frac{\sqrt{\alpha B_k}}{2}\right].$$

In the last inequality, the term $(3d+1)$ is due to the fact that the martingale is stopped at time $S$ and if it goes above the upper limit $3d$ then its overshoot is at maximum the value of the last jump of $M^{T(k)}_{.\wedge S}$ which is of the form $1/w_i$ for some $i\geq X_k$ and is therefore smaller than $\sqrt{B_k}$. Then, by using $\mathbb{P}\left[M^{T(k)}_{S}\geq\sqrt{\alpha B_k}/2\right]=1-\mathbb{P}\left[M^{T(k)}_{S}<\sqrt{\alpha B_k}/2\right]$ we deduce that

$$\mathbb{P}\left[M^{T(k)}_{S}<\frac{\sqrt{B_k}}{20}\right]<\frac{3d+1-\sqrt{\alpha}}{3d+1-\sqrt{\alpha}/2}<1.$$

This means that with positive probability uniformly bounded away from 0, either $M^{T(k)}$ always stays above $\sqrt{\alpha B_k}/2$ or it goes above $3d\sqrt{B_k}$ in a finite time. In this last case the Tchebytchev inequality shows that there is also a positive probability for the martingale to stay above $\sqrt{\alpha B_k}/2$ eventually. Indeed

$$\text{Var}(\left.M_\infty-M_S\,\right|\,\mathcal{F}_S)=\mathbb{E}\left[\left.(M_\infty-M_S)^2\right|\mathcal{F}_S\right]=\mathbb{E}\left[\left.\sum_{\ell=S}^\infty\mathbb{E}\left[\left.(M_{\ell+1}-M_\ell)^2\right|\mathcal{F}_\ell\right]\right|\mathcal{F}_S\right]<2d{B_k}.$$

So we come to the conclusion that with positive probability uniformly bounded away from 0 we have $M_\infty^{T(k)}\geq \sqrt{\alpha B_k}/2>\frac{2d}{w_{X_k}}$ and since $\left|N_\infty - M^{T(k)}_\infty\right|\leq \frac{2d}{w_{X_k}}$, we can conclude that $N_\infty\neq 0$. This ends the proof in the first case.

It now remains to prove Theorem \ref{theoreme} in the second case. The proof is to compare with Corollary 3 in \cite{limictarres}. Since $\liminf_{k\to\infty}\frac{1}{w_{X_n}^2 B_n}>0$, there exists $\varepsilon>0$ and a finite time $k_0$ such that for $k\geq k_0$,
$$\frac{1}{w_{k}^2}\geq \varepsilon\sum_{i=k}^\infty \frac{1}{w_i^2}.$$
This implies, for all $k\geq k_0$, that
\begin{eqnarray*}
\frac{1}{w_k^2}&\geq&\varepsilon\sum_{i=k}^\infty \frac{1}{w_i^2}\geq \varepsilon^2\sum_{i=k}^\infty \sum_{j=i}^\infty\frac{1}{w_j^2}\geq\varepsilon^2\sum_{j=k}^\infty \frac{j-k+1}{w_j^2}\\
&\geq&\varepsilon^3\sum_{j=k}^\infty \sum_{\ell=j}^\infty \frac{j-k+1}{w_\ell^2}\geq\frac{\varepsilon^3}{2}\sum_{\ell=k}^\infty\frac{(\ell-k+1)^2}{w_{\ell}^2}.
\end{eqnarray*}

Then using the Cauchy-Scharz inequality, for all $k\geq k_0$,
\begin{eqnarray*}
\sum_{\ell=k}^\infty\frac{1}{w_\ell}&=&\sum_{\ell=k}^\infty\frac{\ell-k+1}{w_\ell}\frac{1}{\ell-k+1}\\
&\leq&\sqrt{\sum_{\ell=k}^\infty\frac{(\ell-k+1)^2}{w_\ell^2}}\sqrt{\sum_{\ell=k}^\infty\frac{1}{(\ell-k+1)^2}}\\
&\leq&\frac{\pi}{\sqrt{6}}\sqrt{\frac{2}{\varepsilon^3}}\frac{1}{w_k}.
\end{eqnarray*}

The value of the constant on the RHS does not matter here, what is important is the fact that $\limsup_{k\to \infty}w_k \sum_{\ell=k}^\infty\frac{1}{w_\ell} <\infty$. To conclude, let us calculate the probability for any $k\geq k_0$ that the urn always draws the majority colour after time $k$. Denote by $B$ this last event, then

$$\mathbb{P}\left[B\right]\geq\prod_{\ell=k}^\infty \left(\frac{w_\ell}{w_\ell+w_k}\right)^d=\left(\prod_{\ell=k}^\infty \left(1-\frac{w_k}{w_\ell+w_k}\right)\right)^d\geq\left(\prod_{\ell=k}^\infty \left(1-\frac{w_k}{2 w_\ell}\right)\right)^d.$$

The last quantity is bounded away from 0, uniformly in $k$, since

$$\sum_{\ell=k}^\infty\frac{w_k}{w_\ell}\leq \frac{\pi}{\sqrt{6}}\sqrt{\frac{2}{\varepsilon^3}}<\infty.$$

Therefore, the L\'evy 0-1 law implies that with probability one the urn always draws the same colour after a finite time. This ends the proof of the second case and of Theorem \ref{theoreme}.

\bibliographystyle{plain}
\bibliography{biblio}

\end{document}